\documentclass[12pt, a4paper]{amsart}

\usepackage[utf8]{inputenc}
\usepackage{amsmath}
\usepackage{amssymb}
\usepackage{amsthm}
\usepackage{amsaddr}
\usepackage{enumitem}
\usepackage{tikz-cd}
\usepackage{geometry}
\usepackage{titlesec}
\usepackage{mwe}

\titleformat{\section}{\centering\scshape\mdseries}{\thesection.}{0.5em}{}
\titleformat{\subsection}[runin]{\normalfont\bfseries}{\thesubsection.}{0.5em}{}[. ]
\titlespacing{\subsection}{0pt}{1.5ex plus .2ex}{0.5ex plus .1ex}

\theoremstyle{definition}
\newtheorem{definition}{Definition}[section]
\newtheorem{example}{Example}[section]
\newtheorem{remark}{Remark}[section]
\newtheorem*{ml_theorem-proof}{Proof of theorem \ref{ml_theorem}}
\newtheorem*{suspension-flexible-proof}{Proof of $G$-orbit decomposition theorem}

\theoremstyle{plain}
\newtheorem{lemma}{Lemma}[section]
\newtheorem{theorem}{Theorem}[section]
\newtheorem{corollary}{Corollary}[section]
\newtheorem{proposition}{Proposition}[section]
\newtheorem*{suspension-flexible}{$G$-orbit decomposition theorem}

\newcommand{\Affine}{\mathbb{A}}
\newcommand{\VV}{\mathbb{V}}
\newcommand{\NN}{\mathbb{N}}
\newcommand{\KK}{\mathbb{K}}
\newcommand{\KKX}{\mathbb{K}[X]}
\newcommand{\FKX}{\mathbb{K}(X)}

\newcommand{\Aut}{\operatorname{Aut}}
\newcommand{\SAut}{\operatorname{SAut}}
\newcommand{\Ker}{\operatorname{Ker}}
\newcommand{\Der}{\operatorname{Der}}
\newcommand{\LND}{\operatorname{LND}}
\newcommand{\Frac}{\operatorname{Frac}}
\newcommand{\Spec}{\operatorname{Spec}}
\newcommand{\ML}{\operatorname{ML}}
\newcommand{\FML}{\operatorname{FML}}
\newcommand{\Susp}{\operatorname{Susp}}
\newcommand{\codim}{\operatorname{codim}}
\newcommand{\rk}{\operatorname{rk}}
\newcommand{\reg}{\operatorname{reg}}

\newcommand{\Ga}{\mathbb{G}_a}
\newcommand{\Df}{\mathcal{D}_f}
\newcommand{\If}{\mathcal{I}_f}
\newcommand{\Rf}{\mathcal{R}_f}
\newcommand{\Gf}[2]{\mathcal{#1}^{(#2)}_{f}}

\newcommand{\RegY}{\operatorname{Reg}_X Y}
\newcommand{\RegVf}{\operatorname{Reg}_X \mathbb{V}(f)}

\title{\scshape\mdseries the flexibility of $m$-suspensions constructed via a local slice}
\author{\scshape\mdseries isaev roman}
\address{Lomonosov Moscow State University, Faculty of Mechanics and Mathematics, Department of Higher Algebra, Leninskie Gory 1, Moscow, 119991, Russia}
\date{}
\thanks{The work was supported by the Theoretical Physics and Mathematics Advancement Foundation "BASIS", project number 24-7-2-20-1.}

\begin{document}
\maketitle

\begin{quote}
\small
{\scshape\mdseries Abstract.} We refer to the variety $\Susp(X, f, k_1, \dots, k_m) = \VV(y_1^{k_1} \dots y_m^{k_m} - f(x)) \subset X \times \Affine^m$ as an \textit{$m$-suspension over affine variety $X$, constructed via a local slice $f(x) \in \KKX$}, if there exists a locally nilpotent derivation $\delta$ on X such that $\delta (f) \neq 0, \delta^2 (f) = 0$. In this paper, we determine the sufficient conditions under which such a variety is generically flexible and those under which it is flexible. Furthermore, for a flexible $X$ we propose a construction of a local slice $f$ that guarantees the flexibility of the suspension $\Susp(X, f, 1, k_2, \dots, k_m)$.
\end{quote}

\vspace{\baselineskip}

\section{Introduction}

Let $X$ be an irreducible affine algebraic variety over an algebraically closed field $\KK$ of characteristic zero. By $\Aut(X)$ we denote the group of regular automorphisms of the variety $X$. \emph{The special automorphism group} $\SAut(X)$ is the subgroup of $\Aut(X)$ generated, as an abstract group, by all one-parameter unipotent subgroups of $\Aut(X)$.

The action of $\SAut(X)$ on $X$ is closely related to the notion of flexibility. A smooth point of $X$ is called \emph{flexible} if its tangent space is spanned by the tangent vectors to the orbits of one-parameter unipotent subgroups passing through that point. Accordingly, the variety $X$ is called \emph{flexible} if every smooth point of $X$ is flexible. The variety $X$ is said to be \emph{generically flexible} if the set of its flexible points forms an open $\SAut(X)$-orbit. Clearly, every flexible variety is generically flexible, whereas the converse does not hold in general.

The flexibility of a variety is equivalent to the transitivity of the action of $\SAut(X)$ on its smooth locus. Establishing such transitivity, however, may be technically challenging. In contrast, generical flexibility is often easier to prove and, moreover, admits a reformulation in terms of invariants. Recall that the subalgebra of $\SAut(X)$-invariant elements of $\KKX$ is called the \emph{Makar--Limanov invariant} of $X$ and is denoted by $\ML(X)$. Considering rational $\SAut(X)$-invariants analogously leads to the definition of the \emph{field Makar--Limanov invariant}, denoted by $\FML(X)$. In these terms, the criterion for the generical flexibility of $X$ is simply $\FML(X) = \KK$.

Besides providing a criterion for establishing (or disproving) generical flexibility, the invariants $\ML(X)$ and $\FML(X)$ are valuable tools for distinguishing algebraic varieties up to isomorphism. For example, in the pioneering work of Leonid Makar-Limanov \cite{Makar-Limanov}, where the Makar--Limanov invariant was introduced, $\ML(X)$ was used to prove that the Koras–Russell cubic threefold is not isomorphic to $\Affine^3$.

Recall the definition of an \emph{$m$-suspension} over an affine algebraic variety $X$.

\begin{definition}
Fix a positive integer $m$. Let $f \in \KK[X] \setminus \KK$, and let $k_1, \dots, k_m$ be positive integers. Define the affine algebraic variety
\[
Y = \Susp(X, f, k_1, \dots, k_m) = \VV(y_1^{k_1} y_2^{k_2} \dots y_m^{k_m} - f(x)) \subset X \times \Affine^m,
\]
called the \emph{$m$-suspension} over $X$ with weights $k_1, \dots, k_m$.
\end{definition}

The question of the flexibility of the above construction was partially solved in \cite{Arzhantsev-suspensions}. Three classes of affine algebraic varieties were investigated there, and the transitivity of the action of the special automorphism group was established for each of them. One of the constructions considered was the ordinary \emph{suspension} $\Susp(X, f)$. In our terminology, it can be viewed as the $2$-suspension $\Susp(X, f, 1, 1)$. A natural question is whether an arbitrary $m$-suspension is flexible and under what assumptions. Certain properties of $m$-suspensions and their Makar--Limanov invariants were studied in \cite{Gaifullin-suspensions}.

The aim of this paper is to study the construction of the $m$-suspension $Y = \Susp(X, f, k_1, \dots, k_m)$ with arbitrary weights $k_1, \dots, k_m$. More specifically, we investigate the relationship between the invariants $\ML(Y)$ and $\ML(X)$ and formulate sufficient conditions for the flexibility and generical flexibility of $Y$. Many natural examples of algebraic varieties can be realized as $m$-suspensions over other varieties, and even over other $m$-suspensions. Sufficient conditions for the flexibility of $Y$ in terms of $X$ eliminate the need to analyze the action of the special automorphism group on $Y$ directly, allowing one instead to focus on the structurally simpler variety $X$.

Section \ref{sec_preliminaries} introduces all definitions required throughout the paper. Most of them are taken from \cite{Arzhantsev-flexible}, where the notions of flexibility, transitivity, and infinite transitivity of the action of $\SAut(X)$ are studied in detail. In Section \ref{sec_preliminaries_suspensions}, we investigate a class of locally nilpotent derivations on an $m$-suspension over $X$ obtained by extending locally nilpotent derivations from $X$. Theorem \ref{ml_theorem} establishes an upper bound for the Makar--Limanov invariant of an $m$-suspension. It also motivates the study of a new class of objects, namely \emph{$m$-suspensions constructed via a local slice}. Finally, Section \ref{sec_preliminaries_suspensions_ls} concludes our investigation by generalizing the proof from \cite{Arzhantsev-suspensions} to the case of $\Susp(X, f, 1, k_2, \dots, k_m)$ with arbitrary weights $k_2, \dots, k_m$ under an additional assumption on $f \in \KK[X]$.

\section{Preliminaries}\label{sec_preliminaries}

Throughout this paper, all algebraic varieties are assumed to be irreducible, affine, and defined over an algebraically closed field $\KK$ of characteristic zero. We denote by $\Ga$ the additive group of the field $\KK$. We assume that the set of natural numbers $\NN$ contains zero.

\subsection{Basic definitions and notation}\label{ssec_preliminaries_basic}

A subgroup $H$ of the automorphism group $\Aut(X)$ of an algebraic variety $X$ is called an \emph{algebraic} if it admits the structure of an algebraic group such that the natural map $H \times X \rightarrow X$ is a morphism of algebraic varieties. Algebraic subgroups of $\Aut(X)$ isomorphic to $\Ga$ are called \emph{$\Ga$-subgroups}. They are also commonly referred to as \emph{one-parameter unipotent subgroups} of $\Aut(X)$.

A subgroup $G$ of the automorphism group $\Aut(X)$ is called \emph{algebraically generated} if it is generated, as an abstract group, by a family of connected algebraic subgroups of $\Aut(X)$. The subgroup of $\Aut(X)$ algebraically generated by all $\Ga$-subgroups is called the \emph{special automorphism group of $X$} and is denoted by $\SAut(X)$.

The set of smooth points of a variety $X$ will be denoted by $X_{\reg}$. A point $p \in X_{\reg}$ is called \emph{flexible} if the tangent space $T_pX$ is spanned by tangent vectors to the orbits $H.x$ of one-parameter unipotent subgroups $H \subseteq \Aut(X)$. An algebraic variety $X$ is called \emph{flexible} if all of its smooth points are flexible.

Strictly speaking, the notion of flexibility is defined not only for the group $\SAut(X)$ but, more generally, for any algebraically generated group $G$, in which case it is referred to as \emph{$G$-flexibility}. The same applies to several other statements concerning the group $\SAut(X)$: in the literature, they are formulated in this greater generality, namely for arbitrary algebraically generated groups. Since these generalizations are not needed for our purposes, we restrict ourselves to the case of the group $\SAut(X)$.

The following proposition can be found in \cite[Corollary 1.11]{Arzhantsev-flexible}.

\begin{proposition}\label{saut_action}
Consider the action of $\SAut(X)$ on $X$.
\begin{enumerate}[label=\alph*)]
\item A point $p \in X_{\reg}$ is flexible if and only if the orbit $\SAut(X).p$ is open.
\item The open $\SAut(X)$-orbit, if it exists, is unique and contains all flexible points of $X_{\reg}$.
\end{enumerate}
\end{proposition}

Thus, the existence of at least one flexible point guarantees the existence of an open $\SAut(X)$-orbit. Algebraic varieties admitting an open $\SAut(X)$-orbit are called \emph{generically flexible}. Every flexible variety is generically flexible, whereas the converse does not hold in general. Observe that Proposition \ref{saut_action} implies that the flexibility of $X$ is equivalent to the transitivity of the action of $\SAut(X)$ on $X_{\reg}$. A more general theorem, incorporating the notion of infinite transitivity, can be found in \cite[Theorem 0.1]{Arzhantsev-flexible}.

Many examples of varieties $X$ on which $\SAut(X)$ acts transitively are known. The simplest example is $X = \Affine^n$, where $n \geq 1$. The transitivity of the action of $\SAut(\Affine^n)$ is easily verified by considering the locally nilpotent derivations $\frac{\partial}{\partial x_i}$ on $\KK[\Affine^n] = \KK[x_1, \dots, x_n]$ together with the corresponding $\Ga$-subgroups. More will be said about locally nilpotent derivations and $\Ga$-subgroups in Subsection \ref{ssec_preliminaries_LND}. For less elementary examples of flexible varieties, we refer the reader to Section \ref{sec_preliminaries_suspensions_ls}, where the flexibility of $m$-suspensions is discussed.

\subsection{Locally nilpotent derivations}\label{ssec_preliminaries_LND}

Throughout this subsection, we write $A = \KK[X]$ for the coordinate ring of an irreducible affine algebraic variety $X$. Under this assumption, $A$ is an integral domain.

\begin{definition}
A linear map $\delta: A \rightarrow A$ is called a \emph{derivation} if it satisfies the Leibniz rule $\delta(a b) = a \delta(b) + b \delta(a)$.
\end{definition}

\begin{definition}
A derivation $\delta: A \rightarrow A$ is called a \emph{locally nilpotent derivation} (LND) if for every $a \in A$ there exists $n \in \NN$ such that $\delta^n(a) = 0$.
\end{definition}

\begin{remark}
A derivation $\delta$ is locally nilpotent if and only if, for every generator $a_j$ of the algebra $A$, where $a_1, \dots, a_s$ is a generating set of $A$, there exists $n_j$ such that $\delta^{n_j}(a_j) = 0$.
\end{remark}

The sets of all derivations and all locally nilpotent derivations on $A$ are denoted by $\Der(A)$ and $\LND(A)$, respectively. When $A = \KK[X]$, we also write $\Der(X)$ and $\LND(X)$.

The exponential map
\[
\exp(\delta) = \operatorname{id} + \, \delta + \frac{\delta^2}{2!} + \frac{\delta^3}{3!} + \dots \in \Aut(A)
\]
establishes a bijective correspondence between locally nilpotent derivations and $\Ga$-subgroups:
\[
\delta \leftrightarrow \{\exp(s \delta) \; | \; s \in \KK\}.
\]
Conversely, given a $\Ga$-subgroup, the corresponding locally nilpotent derivation is obtained by differentiating its action at the identity:
$\delta(f) = \frac{d (s \cdot f)}{ds} |_{s = 0}$.
Each locally nilpotent derivation $\delta$ defines a vector field, which is also denoted by $\delta$. At every point $p \in X$, this vector field determines the tangent vector to the orbit of the corresponding $\Ga$-subgroup.

\begin{definition}
Let $G$ be a linearly ordered abelian group. A \emph{degree function} is any map $\deg: A \rightarrow G \cup {-\infty}$ satisfying the following conditions for all $a, b \in A$:
\begin{enumerate}
\item $\deg(a) = -\infty \iff a = 0$.
\item $\deg(a b) = \deg(a) + \deg(b)$.
\item $\deg(a + b) \leq \max\{\deg(a), \deg(b)\}$.
\end{enumerate}
\end{definition}

In the above definition, it is understood that $(-\infty) + (-\infty) = -\infty$ and $(-\infty) + g = -\infty$ for every $g \in G$.

Every locally nilpotent derivation $\delta$ induces a degree function on the algebra $A$, defined by
\[
\deg_{\delta}(a) = \min\{n \in \NN \; | \; \delta^{n + 1}(a) = 0\}
\]
for every nonzero $a \in A$. At zero, the function is defined by setting $\deg_{\delta}(0) = -\infty$.

\begin{definition}
Let $S \subset A \setminus {0}$ be a multiplicatively closed subset. The \emph{localization} of $A$ with respect to $S$ is the subring $S^{-1}A \subset \Frac(A)$ defined by
\[
S^{-1} A = \{ab^{-1} \in \Frac(A) \; | \; a \in A, \, b \in S\}.
\]
\end{definition}

We now state several fundamental properties of locally nilpotent derivations, taken from \cite[First Principles for Locally Nilpotent Derivations]{Freudenburg}.

\begin{proposition}
For every $\delta \in \LND(A)$, the kernel $\Ker(\delta)$ is \emph{factorially closed}; that is, if $a, b \in A$ are nonzero and $a b \in \Ker(\delta)$, then $a \in \Ker(\delta)$ and $b \in \Ker(\delta)$.
\end{proposition}

\begin{proposition}\label{cycle_lnd}
Let $\delta \in \LND(A)$ and let $f_1, \dots, f_n \in A$, where $n \geq 1$. Suppose that there exist positive integers $m_1, \dots, m_n$ and a permutation $\sigma \in S_n$ such that $\delta^{m_i} f_i \in f_{\sigma(i)} A$. Then every orbit of $\sigma$ contains an index $i$ satisfying $\delta^{m_i} f_i = 0$.
\end{proposition}

\begin{proposition}\label{replica_lnd}
For $\delta \in \Der(A)$ and nonzero $f \in A$,
\[
f \delta \in \LND(A) \iff \delta \in \LND(A) \, and \, f \in \Ker(\delta).
\]
\end{proposition}

The derivation $f \delta$ is called a \emph{replica} of the derivation $\delta$. Note also that, besides the replica operation, the set $\Der(X)$ is closed under conjugation, that is, under the operation $\delta \mapsto g \delta g^{-1}$ for $g \in \Aut(X)$.

\begin{proposition}\label{localization_lnd}
Let $S \subset A \setminus {0}$ be a multiplicatively closed subset and let $\delta \in \Der(A)$. Then
\[
S^{-1} \delta \in \LND(S^{-1} A) \iff \delta \in \LND(A), \, S \subset \Ker(\delta).
\]
Moreover, if the \emph{localized derivation} $S^{-1} \delta$ belongs to $\LND(S^{-1} A)$, then $\Ker(S^{-1} \delta) = S^{-1} (\Ker , \delta)$.
\end{proposition}

We shall also make extensive use of local slices of derivations.

\begin{definition}
An element $s \in A$ is called a \emph{local slice of $\delta \in \LND(A)$} if $\delta(s) \neq 0$ and $\delta^2(s) = 0$, that is, if $\deg_{\delta}(s) = 1$. In particular, if $\delta(s) = 1$, then $s$ is simply called a \emph{slice}.
\end{definition}

Every nonzero $\delta \in \LND(A)$ admits a local slice. Indeed, it suffices to choose any element $f \in A \setminus \Ker(\delta) \neq \emptyset$ and apply $\delta$ to it exactly $\deg_{\delta}(f)-1$ times.

\begin{definition}
An element $s \in A$ is called a \emph{local slice} (respectively, a \emph{slice}) of $A$ if it is a local slice (respectively, a slice) for some $\delta \in \LND(A)$. When $A = \KK[X]$, we say that $s$ is a local slice (respectively, a slice) of the variety $X$.
\end{definition}

The \emph{Slice Theorem} \cite[Corollary 1.26]{Freudenburg} states that if $s$ is a slice of some $\delta \in \LND(A)$, then $A = (\Ker , \delta)[s]$. In the case $A = \KK[X]$, this can be written as $X \cong \Spec(\Ker , \delta) \times \Affine$. Thus, the existence of a slice on $X$ is equivalent to the existence of an isomorphism $X \cong Z \times \Affine$ for some affine algebraic variety $Z$. The Slice Theorem also implies the following proposition.

\begin{proposition}\label{frac_ker}
Let $\hat{\delta} \in \Der(\Frac(A))$ be the extension of $\delta \in \LND(A)$ to the field $\Frac(A)$. Then $\Ker , \hat{\delta} = \Frac(\Ker , \delta)$.
\end{proposition}

Finally, we state a sufficient condition under which a linear combination of locally nilpotent derivations is again locally nilpotent.

\begin{proposition}\label{lnds_sum}
Let $\delta = a \delta_1 + b \delta_2$ be a linear combination of $\delta_1, \delta_2 \in \LND(A)$ with coefficients $a, b \in A$. Clearly, $\delta$ is a derivation of $A$. Suppose that the derivations commute, $[\delta_1, \delta_2] = 0$, and that the coefficients satisfy $a \in \Ker , \delta_1 \cap \Ker , \delta_2$ and $b \in \Ker \delta_2$. Then $\delta \in \LND(A)$.
\end{proposition}

\begin{proof}
We may assume that $a, b \neq 0$, since otherwise the local nilpotence of $\delta$ follows immediately from Proposition \ref{replica_lnd}. The derivation $\delta_1^{'} = a \delta_1$ is locally nilpotent and satisfies $[\delta_1^{'}, \delta_2] = 0$. Hence, without loss of generality, we may assume that $a = 1$. Let $k = \deg_{\delta_1} b \in \mathbb{N}$.

It is easy to verify that, upon expanding $(\delta_1 + b \delta_2)^n$, every summand is of the form
\[
C b^{r_0} (\delta_1 b)^{r_1} \dots (\delta_1^k b)^{r_k} \delta_1^i \delta_2^j, \;\; C, i, j, r_0, \dots, r_k \in \mathbb{N}
\]
where the exponents satisfy
\[
\sum_{l = 0}^k r_l = j, \;\; i + j + \sum_{l = 0}^k l r_l = n.
\]
These conditions imply the inequality $n \leq i + (k + 1)j$, which yields the estimate
\[
\frac{n}{k + 1} \leq \frac{i}{k + 1} + j \leq i + j \leq n.
\]

Now let $f \in A$ be arbitrary, and choose $m$ such that $\delta_1^m f = \delta_2^m f = 0$. Setting $n = 2m(k + 1)$, we obtain the lower bound $2m \leq i + j$. Therefore, regardless of the values of $i$ and $j$, at least one of them must be greater than or equal to $m$. Consequently, $\delta_1^i \delta_2^j f = 0$, so every summand vanishes.
\end{proof}

\subsection{Makar--Limanov invariant}\label{ssec_preliminaries_Makar-Limanov} 

\begin{definition}
The intersection of the kernels of all locally nilpotent derivations on $X$ is called the \emph{Makar--Limanov invariant}.
\end{definition}

The Makar--Limanov invariant was introduced in \cite{Makar-Limanov}. The invariant $\ML(X)$ coincides with the subalgebra of regular $\SAut(X)$-invariants $\KK[X]^{\SAut(X)}$. The field of rational $\SAut(X)$-invariants $\FKX^{\SAut(X)}$ is denoted by $\FML(X)$ and is called the \emph{field Makar--Limanov invariant}. It can be expressed as the intersection of the kernels of the derivations obtained by extending locally nilpotent derivations from $\KK[X]$ to $\FKX$. Applying Proposition \ref{frac_ker}, we obtain
\[
\FML(X) = \bigcap_{\delta \, \in \, \LND(X)} \Frac(\Ker \, \delta).
\]

The following criterion for generical flexibility in terms of the Makar--Limanov field invariant is given in \cite[Proposition 5.1]{Arzhantsev-flexible}.

\begin{proposition}\label{criteria_fml}
A variety $X$ is generically flexible if and only if $\FML(X) = \KK$.
\end{proposition}

The ordinary Makar--Limanov invariant is not sufficient for this purpose. Indeed, if $X$ is generically flexible, then $\ML(X) = \KK$, whereas the converse does not hold in general. Proposition \ref{criteria_fml} is, in fact, a consequence of the more general result \cite[Corollary 1.14]{Arzhantsev-flexible}.


\begin{proposition}\label{criteria_G}
Let $G$ be an algebraically generated group acting on $X$. Then $G$ has an open orbit if and only if $\FKX^G = \KK$.
\end{proposition}

\subsection{Additional definitions}\label{ssec_preliminaries_additional}

Let $G \subseteq \Aut(X)$ be a subgroup. The vector fields in $\LND(X)$ generating one-parameter unipotent subgroups of $G$ form a subset $\LND(G) \subseteq \LND(X)$. This subset is closed under replicas $\delta \mapsto f \delta$ for $f \in \Ker \, \delta$ and under conjugation $\delta \mapsto g \delta g^{-1}$ for $g \in G$.

Let $G \subseteq \SAut(X)$ be a $\Ga$-generated subgroup. Consider a subset $\mathcal{N} \subseteq \LND(G)$ such that the corresponding family of $\Ga$-subgroups
\[
\mathcal{H} = \{\exp(\KK \delta) \, | \, \delta \in \mathcal{N}\}
\]
generates $G$. The set $\mathcal{N}$ is called a \emph{generating set} of $G$, and we write $G = \langle \mathcal{N} \rangle$.

\begin{proposition}
Let $\mathcal{N} \subseteq \LND(X)$ and $G = \langle \mathcal{N} \rangle$. Then
\[
\bigcap_{\delta \, \in \, \mathcal{N}} \Ker \, \delta = \KK[X]^G.
\]
\end{proposition}

\begin{proof}
Let $\mathcal{H}$ denote the family of $\Ga$-subgroups corresponding to the set $\mathcal{N} \subseteq \LND(X)$. It is well known that, for every $\Ga$-subgroup $H \in \mathcal{H}$, the algebra of $H$-invariants $\KK[X]^H$ coincides with the kernel $\Ker \, \delta$ of the corresponding locally nilpotent derivation $\delta$. A polynomial invariant under the action of every subgroup $H \in \mathcal{H}$ is invariant under the action of the group $G$ generated by the family $\mathcal{H}$. The converse is immediate. Therefore,
\[
\bigcap_{\delta \, \in \, \mathcal{N}} \Ker \, \delta = \bigcap_{H \, \in \, \mathcal{H}} \KK[X]^H = \KK[X]^G.
\]
\end{proof}

Let $n$ be a nonnegative integer and let $f \in \KK[X]$. Introduce the notation
\[
\Gf{N}{n} = \{\delta \in \LND(X) \, | \, \deg_{\delta} f = n \} \cup \{0\}.
\]
The corresponding group is denoted by $\Gf{G}{n} = \langle \Gf{N}{n} \rangle$.

\begin{proposition}
For every $n \geq 0$, the family $\mathcal{N}^{(n)}_f$ is invariant under conjugation by elements of $\mathcal{G}^{(0)}_f$.
\end{proposition}

\begin{proof}
Indeed, for every $g \in \mathcal{G}^{(0)}*f$, we have $g f = f$. On the other hand, $g \delta g^{-1} \in \LND(X)$ for every $\delta \in \LND(X)$. It remains to verify that the degree is preserved under conjugation:
\[
n = \deg_{\delta} f = \deg_{g \delta g^{-1}} gf = \deg_{g \delta g^{-1}} f.
\]
\end{proof}

\section{$m$-suspensions}\label{sec_preliminaries_suspensions}

\begin{definition}
Fix a positive integer $m$. Let $f \in \KK[X] \setminus \KK$, and let $k_1, \dots, k_m$ be positive integers. Define the affine algebraic variety
\[
Y = \Susp(X, f, k_1, \dots, k_m) = \mathbb{V}(y_1^{k_1} y_2^{k_2} \dots y_m^{k_m} - f(x)) \subset X \times \Affine^m,
\]
called the \emph{$m$-suspension} over $X$ with weights $k_1, \dots, k_m$.
\end{definition}

Such a variety $Y$ may be reducible; for example, consider $\Susp(X, x^2, 2, \dots, 2)$. In certain special cases, the question of irreducibility has been resolved. In particular, $m$-suspensions over the affine line ($X \cong \Affine$) have the irreducibility criterion established in \cite{Danielewski}, while $m$-suspensions of the form $\Susp(X, f, 1, \dots, 1)$ over an irreducible variety $X$ are themselves irreducible, as follows from \cite[Lemma 3.1]{Arzhantsev-suspensions}. In the general case, however, it is difficult to guarantee irreducibility by imposing only assumptions on $f(x)$. For this reason, throughout the remainder of the paper we assume that $Y$ is irreducible, an assumption required in the proofs of many of the subsequent results.

In this section, we describe a class of locally nilpotent derivations on an $m$-suspension obtained from locally nilpotent derivations on the base variety $X$. We derive an upper bound for the Makar--Limanov invariant of $m$-suspensions and establish several important consequences.

For convenience, we use multi-index notation: if $\alpha = \sum_{i = 1}^m \alpha_i e_i \in \mathbb{N}^m$, then $y^{\alpha} = y_1^{\alpha_1} \dots y_m^{\alpha_m}$. The vector $k = \sum_{i = 1}^m k_i e_i$ will be referred to as the \emph{weight vector} of the $m$-suspension.

\subsection{Extension of LNDs}

Let $\KK[X]$ be generated by the elements $x_1, \dots, x_s$. Given $\delta \in \LND(X)$, it extends naturally to $\KK[X \times \Affine^m] \cong \KK[X] \otimes \KK[y_1, \dots, y_m]$ as the derivation $\delta^{'}$ defined by
\[
\delta^{'}(x_j) = \delta(x_j), \;\; j = 1, \dots, s, \;\;\;\; \delta^{'}(y_l) = 0, \;\; l = 1, \dots, m.
\]
Using $\delta^{'}$, we construct a new locally nilpotent derivation on $X \times \Affine^m$ that descends to the coordinate ring of the $m$-suspension
\[
\KK[Y] = (\KKX \otimes \KK[y_1, \dots, y_m]) / (y^k - f(x)).
\]
In what follows,
\[
(\Ker \, \delta^{'})_i = (\Ker \, \delta)[y_1, \dots, \hat{y}_{i} \dots, y_m]
\]
denotes the algebra over $\Ker \, \delta$ generated by the variables $y_l$, where $l = 1, \dots, m$ and $l \neq i$.

\begin{definition}\label{lnd_deff}
Let $\delta \in \LND(X)$.

\begin{enumerate}
    \item If $\delta(f) \neq 0$, then $\delta$ is associated with the derivation $\hat{\delta}_i = \hat{\delta}_i(\delta, q) \in \Der(X \times \Affine^m)$ defined by
    \[
    \hat{\delta}_i(\delta, q) = q k_i y^{k - e_i} \delta^{'} + q \delta^{'}(f) \partial_{y_i},
    \]
    where $i \in \{1, \dots, m\}$ is fixed and $q \in (\Ker \, \delta^{'})_i$ is arbitrary.

    \item If $\delta(f) = 0$, then $\delta$ is associated with the derivation $\hat{\delta}(\delta, q) = q \delta^{'} \in \Der(X \times \Affine^m)$, where $q \in \Ker \, \delta^{'} = (\Ker \, \delta)[y_1, \dots, y_m]$ is arbitrary.
\end{enumerate}

\end{definition}

For the sake of uniform notation, we shall occasionally write $\hat{\delta}(\delta, q)$ as $\hat{\delta}_i(\delta, q)$ for an arbitrary index $i$. Which of the two types of derivations is meant by the notation $\hat{\delta}_i(\delta, q)$ will always be clear from the context, depending on whether $\delta(f) \neq 0$ or $\delta(f) = 0$.

A direct computation shows that derivations of both types preserve the ideal $(y^k - f(x))$ and therefore descend to derivations on $Y$. We denote the induced quotient derivations by the same symbols $\hat{\delta}(\delta, q)$ and $\hat{\delta}_i(\delta, q)$. Their local nilpotence on $Y$ follows from the local nilpotence of $\hat{\delta}(\delta, q)$ and $\hat{\delta}_i(\delta, q)$ on $X \times \Affine^m$. Clearly, $\hat{\delta}(\delta, q) \in \LND(X \times \Affine^m)$. We now formulate a criterion for the local nilpotence of $\hat{\delta}_i(\delta, q)$ on $X \times \Affine^m$.

\begin{lemma}\label{lnd_criteria}
Let $\delta \in \LND(X)$ satisfy $\delta(f) \neq 0$, and fix $i \in {1, \dots, m}$.

\begin{enumerate}
    \item If $k_i = 1$, then $\hat{\delta}_i(\delta, q) \in \LND(X \times \Affine^m)$.
    \item If $k_i > 1$, then $\hat{\delta}_i(\delta, q) \in \LND(X \times \Affine^m) \iff \delta^2(f) = 0$.
\end{enumerate}

\end{lemma}

\begin{proof}
In both parts, we use the simple observation that the derivations $\delta^{'}$ and $\partial_{y_i}$ commute. We also note that the conditions $\delta^{'2}(f) = 0$ and $\delta^2(f) = 0$ are equivalent.

1. Suppose that $k_i = 1$. Then
\[
q k_i y^{k - e_i} \in \Ker \, \delta^{'} \cap \Ker \, \partial_{y_i}, \;\;\;\; q \delta^{'} (f) \in \Ker \, \partial_{y_i}.
\]
Proposition \ref{lnds_sum} implies that $\hat{\delta}_i(\delta, q) \in \LND(X \times \Affine^m)$.

2. Assume that $k_i > 1$. If $\delta^{'2}(f) = 0$, then
\[
q k_i y^{k - e_i} \in \Ker \, \delta^{'}, \;\;\;\; q \delta^{'} (f) \in \Ker \, \delta^{'} \cap \Ker \, \partial_{y_i}.
\]
Proposition \ref{lnds_sum} again yields $\hat{\delta}_i(\delta, q) \in \LND(X \times \Affine^m)$.

Conversely, suppose that $\hat{\delta}_i(\delta, q) \in \LND(X \times \Affine^m)$. Differentiating $y_i$ twice, we obtain
\[
\hat{\delta}_i (y_i) = q \delta^{'} (f), \;\;\;\; \hat{\delta}^2_i (y_i) = q^2 k_i y^{k - e_i} \delta^{'2}(f).
\]
Since $\hat{\delta}_i^2(y_i)$ belongs to the ideal $(y^{k - e_i}) \subset (y_i)$, proposition \ref{cycle_lnd} implies that $\hat{\delta}_i^2(y_i) = 0$. Consequently, $\delta^{'2}(f) = 0$.
\end{proof}

\begin{remark}\label{new_lnd_ls}
If $f$ is a local slice of $\delta \in \LND(X)$, then $y_i$ is a local slice of $\hat{\delta}_i(\delta, q)$. Moreover, if $k_i = 1$, then $f$ is also a local slice of $\hat{\delta}_i(\delta, q)$. This simple observation allows to construct a new $m$-suspension over the current one using the same local slice $f$.
\end{remark}

\subsection{Makar--Limanov invariant of an $m$-suspension}

\begin{definition}
For $f \in \KK[X] \setminus \ML(X)$, define the subalgebra
\[
\ML_f(X) = \bigcap_{\substack{\delta \, \in \, \LND(X), \\ \delta (f) \neq 0}} \Ker \, \delta.
\]
\end{definition}

Let $\pi: \KK[X \times \Affine^m] \rightarrow \KK[Y]$ denote the quotient map $\pi(h) = h + I$, where $I = (y^k - f(x))$. In what follows, the expression $\pi(\ML_f(X))$ is understood via the natural embedding of $\ML_f(X)$ into $\KK[X \times \Affine^m]$.

\begin{theorem}\label{ml_theorem}
Let $Y = \Susp(X, f, k_1, \dots, k_m)$.

\begin{enumerate}
\item If $\forall \, i: k_i = 1$ and $f \in \KK[X] \setminus \ML(X)$, then $\ML(Y) \subset \pi(\ML_f(X))$.
\item If $\exists \, i, j: k_i = 1, \, k_j > 1$ and $f$ is a local slice of $X$, then $\ML(Y) \subset\pi(\ML_f(X))$.
\item If $\forall \, i: k_i > 1$ and $f$ is a local slice of $X$, then $\ML(Y) \subset \pi(\KKX^{\Gf{G}{1}})$.
\end{enumerate}

\end{theorem}

The proof will be given later. We first consider several important special cases.

\begin{corollary}
Let $Y = \Susp(X, f, k_1, \dots, k_m)$, where $f \in \KK[X]$. Suppose that $\Gf{N}{0} = {0}$ and one of the following conditions holds:
\begin{enumerate}
\item $\forall \, i: k_i = 1$;
\item $\exists \, i, j: k_i = 1$, $k_j > 1$, and $f$ is a local slice of $X$.
\end{enumerate}
Then $\ML(Y)$ is contained in $\pi(\ML(X))$. Moreover, if $X$ is generically flexible, then $\ML(Y)$ is trivial.
\end{corollary}

\begin{proof}
The condition $\Gf{N}{0} = {0}$ means that there is no $\delta \in \LND(X)$ satisfying $\delta(f) = 0$. Under this assumption, $\ML_f(X) = \ML(X)$. If $X$ is generically flexible, then $\ML(X)$ is trivial. Hence $\ML(Y) = \KK$.
\end{proof}

\begin{corollary}
Let $f$ be a local slice of $X$, and suppose that $\Gf{G}{1}$ acts on $X$ with an open orbit. Then the suspension $Y = \Susp(X, f, k_1, \dots, k_m)$ has trivial Makar--Limanov invariant, regardless of the weights $k_1, \dots, k_m$.
\end{corollary}

\begin{proof}
We have the inclusion
\[
\ML_f(X) = \bigcap_{\substack{\delta \, \in \, \LND(X), \\ \delta (f) \neq 0}} \Ker \, \delta \subseteq \bigcap_{\delta \, \in \, \Gf{N}{1}} \Ker \, \delta = \KKX^{\Gf{G}{1}}.
\]
By Proposition \ref{criteria_G} $\,\KKX^{\Gf{G}{1}} = \KK$.
\end{proof}


\begin{definition}
If $f$ is a local slice of $X$, then $\Susp(X, f, k_1, \dots, k_m)$ is called an \emph{$m$-suspension constructed via the local slice $f$}.
\end{definition}

\subsection{Auxiliary lemmas}

To prove Theorem \ref{ml_theorem}, we first establish several auxiliary lemmas. Recall the notation
\[
(\Ker \, \delta^{'})_i = (\Ker \, \delta)[y_1, \dots, \hat{y}_{i} \dots, y_m].
\]

\begin{lemma}\label{cont_lnd_ker}
Let $\delta \in \LND(X)$ satisfy $\delta(f) \neq 0$ and $\delta^2(f) = 0$. Then the derivation $\hat{\delta}_i(\delta, q) \in \LND(Y)$ is well defined, and its kernel is given by
\[
\Ker \, \hat{\delta}_i = \pi((\Ker \, \delta^{'})_i).
\]
\end{lemma}

\begin{proof}
The inclusion from right to left is immediate. We prove the converse. Consider the localizations $S^{-1}\KK[Y]$ and $S^{-1}\KK[X \times \Affine^m]$ with respect to the multiplicative subset $S$ generated by ${y_1, \dots, \hat{y}_i, \dots, y_m}$. For clarity, we display the relationships among the four algebras in the following commutative diagram:
\[
\begin{tikzcd}[sep=large]
\mathbb{K}[X \times \Affine^m] \arrow[r, hook] \arrow{d}{\pi} & S^{-1}\mathbb{K}[X \times \Affine^m] \arrow{d}{\pi} \\
\mathbb{K}[Y] \arrow[r, hook] & S^{-1}\mathbb{K}[Y]
\end{tikzcd}
\]
Here, the quotient map on the right is obtained by extending the left-hand map $\pi$ to $S^{-1}\KK[X \times \Affine^m]$, while the horizontal arrows are the canonical localization embeddings.

As throughout the paper, we assume that $Y$ is irreducible. Since $X$ is irreducible, so is $X \times \Affine^m$. Hence both coordinate rings on the left are integral domains, and therefore so are their localizations. Because $S \subset \Ker \, \hat{\delta}_i$, the derivation $\hat{\delta}_i$ extends to a locally nilpotent derivation on both localized algebras.

The algebra $S^{-1}\KK[Y]$ admits a $\mathbb{Z}_{k_i}$-grading
\[
S^{-1}\KK[Y] = \bigoplus_{l = 0}^{k_i - 1} A_l,
\]
where the grading is determined by the degree in the variable $y_i$. The localized derivation $S^{-1}\hat{\delta}_i$ is homogeneous of degree $-1$ with respect to this grading. Consequently, its kernel is a graded subalgebra of $S^{-1}\KK[Y]$. Suppose there exists a nonzero element $h \in A_l \cap \Ker \, S^{-1}\hat{\delta}_i$ with $l \neq 0$. Then $h$ can be written as $h = y_i \hat{h}$. Since the kernel of a locally nilpotent derivation is factorially closed, it follows that $y_i \in \Ker \, S^{-1}\hat{\delta}_i$, contradicting the equality
\[
(S^{-1}\hat{\delta}_i)(y_i) = q \delta^{'}(f) + S^{-1}I \neq S^{-1}I.
\]
Therefore,
\[
\Ker \, S^{-1}\hat{\delta}_i \subseteq A_0.
\]

Let $(h + S^{-1}I) \in \Ker \, S^{-1}\hat{\delta}_i$ be arbitrary. Since $\Ker \, S^{-1}\hat{\delta}_i \subseteq A_0$, we may choose the representative $h$ in $S^{-1}\KK[X \times \Affine^m]$ to be independent of $y_i$. Then
\[
(S^{-1}\hat{\delta}_i)(h + S^{-1}I) = qk_iy^{k - e_i} (S^{-1}\delta^{'})(h) + S^{-1}I = S^{-1}I.
\]
Passing to the localized polynomial ring $S^{-1}(\KK[X \times \Affine^m])$ yields
\[
qk_iy^{k - e_i} (S^{-1}\delta^{'})(h) \in S^{-1}I.
\]
Since we are working in an integral domain, the ideal $S^{-1}I$ is prime. Hence
\[
(S^{-1}\delta^{'})(h) \in S^{-1}I.
\]
Moreover, both $h$ and $(S^{-1}\delta^{'})(h)$ are independent of $y_i$. Expanding
\[
(S^{-1}\delta^{'})(h) = \left(\sum_{l = 0}^{r} a_l y_i^l\right) (y_i^{k_i} - \hat{f}), \;\;\;\; \hat{f} = \frac{f}{y^{k - k_ie_i}}
\]
and comparing the term of highest degree in $y_i$, we conclude that all coefficients $a_l$ vanish. Thus
\[
(S^{-1}\delta^{'}) (h) = 0,
\]
where $h$ is independent of $y_i$.

Passing back to the quotient gives
\[
S^{-1}\Ker \, \hat{\delta}_i = \Ker \, S^{-1}\hat{\delta}_i \subseteq \pi((\Ker \, S^{-1} \delta^{'})_i) = S^{-1} \pi((\Ker \, \delta^{'})_i).
\]
Both equalities follow from Proposition \ref{localization_lnd} describing the kernels of localized locally nilpotent derivations. In the second equality, we implicitly use that localization commutes with the quotient map $\pi$. Finally, intersecting with the original coordinate ring yields
\[
\Ker \, \hat{\delta}_i \subseteq S^{-1} \pi((\Ker \, \delta^{'})_i) \cap \KK[Y] = \pi((\Ker \, \delta^{'})_i),
\]
which completes the proof. It remains only to justify the equality
\[
S^{-1} \pi((\Ker \, \delta^{'})_i) \cap \KK[Y] = \pi((\Ker \, \delta^{'})_i).
\]

Let $p$ be an element of the intersection. Multiplying $p$ by a suitable element $s \in S$, we may assume that
\[
ps \in \pi((\Ker \, \delta^{'})_i).
\]
Choose $h \in (\Ker \, \delta^{'})_i$ such that $ps - h \in I$. Equivalently,
\[
ps - h = (y^k - f)w
\]
for some $w \in \KK[X \times \Affine^m]$.

Now choose a variable $y_j$, where $j \neq i$, appearing in the product $s$, and show that $h$ is divisible by $y_j$. Consider the $\mathbb{Z}$-grading on $\KK[X \times \Affine^m]$ induced by the degree in $y_j$. The LND $\hat{\delta}_i$ is homogeneous of degree zero with respect to this grading, so its kernel is a graded subalgebra of $\KK[X \times \Affine^m]$. Decompose $h$ and $w$ into their homogeneous components with respect to this grading, and let $h_0$ and $w_0$ denote the components of degree zero. From $ps - h = (y^k - f)w$ we obtain
\[
h_0 - fw_0 \in (y_j),
\]
which implies $fw_0 = h_0$.

Since the kernel is graded, $h_0 \in \Ker \, \hat{\delta}_i$. If $h_0 \neq 0$, then the equality
\[
fw_0 = h_0 \in \Ker \, \hat{\delta}_i,
\]
together with the factorial closedness of $\Ker \, \hat{\delta}_i$, implies that
\[
f \in \Ker \, \hat{\delta}_i,
\]
contradicting the assumptions of the lemma. Hence $h_0 = 0$, so $h$ is divisible by $y_j$.

We now divide both $s$ and $h$ by $y_j$ and repeat the argument, relabeling the resulting elements again by $s$ and $h$. After finitely many iterations, we arrive at the case $s = 1$. Then $p - h \in I$, and therefore
\[
p + I = h + I \in \pi((\Ker \, \delta^{'})_i).
\]
This proves the desired equality.
\end{proof}

To study the intersection of the kernels of LNDs extended to $Y$, it is necessary to determine when the equality
\[
\pi(A_1) \cap \dots \cap \pi(A_n) = \pi(A_1 \cap \dots \cap A_n)
\]
holds for subalgebras $A_1, \dots, A_n \subseteq \KK[X \times \Affine^m]$. The inclusion from right to left is immediate, whereas the converse does not hold in general.

\begin{lemma}\label{intersection_lemma}
Let $\mathcal{M}$ be a family of subalgebras of $\KK[X]$ for some variety $X$, and let $\pi: \KK[X] \rightarrow \KK[Y]$ be the quotient map modulo the ideal $I$. Suppose that there exists a fixed subalgebra $A \in \mathcal{M}$ such that $(A + B) \cap I = 0$ for every subalgebra $B \in \mathcal{M}$. Then
\[
\bigcap_{B \, \in \, \mathcal{M}} \pi(B) = \pi\left(\bigcap_{B \, \in \, \mathcal{M}} B\right).
\]
\end{lemma}

\begin{proof}
It suffices to prove the inclusion from left to right. Every element of the left-hand intersection can be represented as $a + I$ for some $a$ belonging to the fixed subalgebra $A \in \mathcal{M}$. Since
\[
a + I \in \pi(B)
\]
for every subalgebra $B \in \mathcal{M}$, there exists an element $b \in B$ such that
\[
a + I = b + I.
\]
Hence $a - b \in I$. Since $a - b \in A + B$, the assumption of the lemma implies that $a - b = 0$ and therefore $a = b \in B$. As $B$ was arbitrary, we conclude that
\[
a \in \bigcap_{B \, \in \, \mathcal{M}} B
\]
whence
\[
a + I \in \pi\left(\bigcap_{B \, \in \, \mathcal{M}} B\right).
\]
\end{proof}

\begin{lemma}\label{frac_ker_lemma}
Let $Y = \Susp(X, f, k_1, \dots, k_m)$.

\begin{enumerate}
    \item Fix $\delta \in \LND(X)$ such that $\delta(f) \neq 0$. If either $\forall \, i: k_i = 1$ or $\delta^2(f) = 0$, then
    \[
    \bigcap_{i = 1}^m \Ker \, \hat{\delta}_i(\delta, q) = \pi(\Ker \, \delta).
    \]

    \item Fix $i$ such that $k_i = 1$. If $f \in \KK[X] \setminus \ML(X)$, then
    \[
    \bigcap_{\substack{\delta \, \in \, \LND(X) \\ \delta (f) \neq 0}} \Ker \, \hat{\delta}_i(\delta, q) =\pi(\ML_f(X)[y_1, \dots, \hat{y}_{i}, \dots, y_m]).
    \]
\end{enumerate}

\end{lemma}

\begin{proof}
We begin by observing that in both cases all derivations $\hat{\delta}_i(\delta, q)$ whose kernels appear in the left-side intersections are locally nilpotent.

1. Since
\[
((\Ker \, \delta^{'})_i + (\Ker \, \delta^{'})_j) \cap I = 0
\]
for arbitrary $i$ and $j$, we may apply Lemma \ref{intersection_lemma} to obtain
\[
\bigcap_{i = 1}^m \Ker \, \hat{\delta}_i(\delta, q) = \bigcap_{i = 1}^m \pi((\Ker \, \delta^{'})_i) = \pi\left(\bigcap_{i = 1}^m (\Ker \, \delta^{'})_i \right) = \pi(\Ker \, \delta).
\]

2. Observe that
\[
((\Ker \, \delta^{'}_1)_i + (\Ker \, \delta^{'}_2)_i) \cap I = 0.
\]
The desired equality now follows from Lemma \ref{intersection_lemma}:
\[
\bigcap_{\substack{\delta \, \in \, \LND(X) \\ \delta (f) \neq 0}} (\Ker \, \delta)[y_1, \dots, \hat{y}_{i} \dots, y_m] = \ML_f(X)[y_1, \dots, \hat{y}_{i}, \dots, y_m].
\]
\end{proof}

\subsection{Proof of the theorem \ref{ml_theorem}}

Now we have all the ingredients to give the proof. Throughout the proof, we implicitly apply Lemma \ref{intersection_lemma}.

1. Applying part (1) of Lemma \ref{frac_ker_lemma} to every $\delta \in \LND(X)$ satisfying $\delta(f) \neq 0$, we obtain
\[
\ML(Y) \subseteq \bigcap_{\substack{\delta \, \in \, \LND(X) \\ \delta (f) \neq 0}} \pi(\Ker \, \delta) = \pi(\ML_f(X)).
\]

2. By the definition of a local slice of $X$, there exists $\delta_f \in \LND(X)$ such that $\delta_f(f) \neq 0$ and $\delta_f^2(f) = 0$. Apply part (1) of Lemma \ref{frac_ker_lemma} to $\delta_f$, and part (2) to an index $i$ with $k_i = 1$. Intersecting the results yields
\[
\ML(Y) \subseteq \pi(\Ker \, \delta_f) \cap \pi(\ML_f(X)[y_1, \dots, \hat{y}_{i}, \dots, y_m]) = \pi(\ML_f(X)).
\]

3. For each $\delta \in \Gf{N}{1}$, apply part (1) of Lemma \ref{frac_ker_lemma} to obtain
\[
\ML(Y) \subseteq \bigcap_{\delta \, \in \, \Gf{N}{1}} \pi(\Ker \, \delta) = \pi\left(\bigcap_{\delta \, \in \, \Gf{N}{1}} \Ker \, \delta\right) = \pi(\KKX^{\Gf{G}{1}}).
\]

\section{$m$-suspensions constructed via a local slice}\label{sec_preliminaries_suspensions_ls}

The flexibility of $m$-suspensions with weights satisfying $\forall \, i : k_i = 1$ was established in \cite[Theorem 3.2]{Arzhantsev-suspensions}.

\begin{theorem}\label{flex_1}
Let $X$ be a flexible irreducible affine variety. Suppose that either $X \cong \Affine$ or $\dim , X \ge 2$. Then the suspension $\Susp(X, f, 1, 1)$ is flexible.
\end{theorem}

Under the assumptions of the theorem, the $m$-suspension $Y = \Susp(X, f, 1, \dots, 1)$ is also flexible. Indeed, it suffices to consider a suitable sequence
\[
X^{(n)} = \Susp(X^{(n - 1)}, f_{n}, 1, 1),
\]
where $X^{(0)} = X$ and $X^{(p)} = Y$ for some $p$, and then apply the theorem successively to each variety $X^{(n)}$. Alternatively, one may observe that the proof of Theorem \ref{flex_1} extends with only minor modifications to the general case of $\Susp(X, f, 1, \dots, 1)$.

In this section, we prove an analogous theorem for $m$-suspensions whose weight vector contains both weights equal to $1$ and weights greater than $1$. In this setting, extending locally nilpotent derivations from $X$ to the $m$-suspension requires assuming that $f \in \KKX$ is a local slice of $X$. Otherwise, for every index $i$ with $k_i > 1$, the derivation $\hat{\delta}_i(\delta, q)$ fails to be locally nilpotent for every $\delta \in \LND(X)$ (see Lemma \ref{lnd_criteria}).

For a fixed index $i$, let $G_{y_i}$ denote the algebraically generated subgroup generated by all $\Ga$-subgroups of the form $\exp(\KK \hat{\delta}_i)$. We further denote by $G$ the subgroup generated by all groups $G_{y_i}$, where $i = 1, \dots, m$, and refer to it as the \emph{group generated by the extended LNDs}. Clearly, $G \subseteq \SAut(Y)$ for the $m$-suspension $Y$.

Using this terminology, we explicitly describe all $G$-orbits in $Y_{\reg}$ under the assumption that $X$ is flexible (see Theorem \ref{orbit_theorem}). In particular, we show that $Y$ is generically flexible. We also formulate sufficient conditions for the flexibility of $Y$ and provide an explicit construction of a local slice $f \in \KKX$ that guarantees the flexibility of the $m$-suspension $\Susp(X, f, 1, k_2, \dots, k_m)$ for arbitrary weights $k_2, \dots, k_m$, whenever $X$ is flexible.

\subsection{Decomposition into $G$-orbits}

\begin{definition}
Let $f$ be a local slice of the variety $X$. Define the subvariety $\Df \subset X$ as $\VV(\If)$ where the ideal $\If$ is generated by the set
\[
\{f\} \cup \{\delta f \, | \, \delta \in \Gf{N}{1} \}.
\]
\end{definition}

\begin{proposition}
The following properties hold.
\begin{enumerate}
\item If $f$ is a slice of $X$, then $\Df = \emptyset$.
\item The subvariety $\Df$ is invariant under the action of $\Gf{G}{0}$ on $X$.
\end{enumerate}
\end{proposition}

\begin{proof}
1. There exists $\delta_f \in \LND(X)$ such that $\delta_f(f) = 1$. Hence $\delta_f \in \Gf{N}{1}$, and therefore $1 \in \If$.

2. This follows immediately from the invariance of $\Gf{N}{1}$ under conjugation by elements of $\Gf{G}{0}$ on $\LND(X)$.
\end{proof}

Now let $J_X(p)$ denote the Jacobian matrix of the variety $X \subseteq \Affine^n$ at a point $p \in X$, and let $J_Y(p)$ denote the Jacobian matrix of the subvariety $Y \subseteq X$ at a point $p \in Y$, obtained from $J_X(p)$ by adjoining the rows corresponding to the defining equations of $Y$.

\begin{definition}
Define $\RegY$ as the set that consists precisely of those points $p \in Y$ for which $\rk J_Y(p) = n - \dim Y$.
\end{definition}

Since $r = n - \dim X$ is the maximal rank of $J_X$, this condition can also be written as $\rk J_Y(p) = r + \codim_X Y$. Let us also notice that $\RegY$ and $Y_{reg}$ do not coincide in general case. One can think of $\RegY$ as of regular points of $Y$ with respect to the fact that $Y$ is a subvariety of $X$.

Regarding the latter notion, our main object of interest will be the set $\RegVf$ for the divisor $\VV(f) \subseteq X$. Its significance will become apparent after the proof of Lemma \ref{lemma_reg}. For now, we only note that
\[
\RegVf \subseteq \VV(f) \cap X_{\reg}.
\]
We also introduce the set
\[
\Rf = \Df \cap \RegVf
\]
which is invariant under the action of $\Gf{G}{0}$ (see Corollary \ref{rf_inv}).

Now, with all these definitions in mind, let us formulate the main result of this section. For the sake of simplicity, we will divide it into two parts - the case of $2$-suspension and the case of $m$-suspension. The general case has little difference with the special one so we will leave it without proof.

In the following theorem, $\Affine^1_u$ and $\Affine^1_v$ denote affine lines with coordinates $u$ and $v$, respectively.

\begin{suspension-flexible}\label{orbit_theorem}
Let $X$ be a flexible affine algebraic variety, and let $f$ be a local slice of $X$. Consider the $2$-suspension
\[
Y = \VV(uv^k - f) \subset X \times \Affine^1_u \times \Affine^1_v
\]
with the weight $k > 1$. Then the group $G$ acts on $Y_{\reg}$ with an open orbit $\mathcal{O}_G$, whose complement is
\[
Y_{\reg} \setminus \mathcal{O}_G = \Rf \times \Affine^1_u \times \{0\}.
\]
Moreover, this complement decomposes into a disjoint union of $G$-orbits in the following way
\[
\bigsqcup_{[p] \, \in \, \Rf / \Gf{G}{0}} \Gf{G}{0} \cdot p \times \Affine^1_u \times \{0\}.
\]
\end{suspension-flexible}

Before giving the formal proof in Subsection \ref{proof_theorem_orbits}, let us explain the underlying idea and the main difficulty.

The variety $Y$ admits a rich family of derivations of the form $\hat{\delta}_u(\delta, q)$, which are locally nilpotent for every $\delta \in \LND(Y)$. In contrast, derivations of the form $\hat{\delta}_v(\delta, q)$ are much more restrictive: they are locally nilpotent only when $\delta \in \Gf{N}{1}$, apart from the derivations $\hat{\delta}(\delta, q)$ with $\delta \in \Gf{N}{0}$. Moreover, at points of the form $(x, u, 0) \in Y$ with $x \in \Df$, the derivations $\hat{\delta}_v(\delta, q)$ degenerate, in the sense that both coefficients in the linear combination
\[
k u v^{k - 1} \delta^{'} + \delta^{'}(f) \partial_v
\]
vanish simultaneously. For these two reasons, the subset
\[
\Rf \times \Affine^1_u \times \{0\}
\]
remains $G$-invariant subset of $X$. Nevertheless, the relatively small collection of locally nilpotent derivations of type $\hat{\delta}_v(\delta, q)$ (in some cases, just a single one, up to taking replicas) is sufficient to ensure the transitivity of the $G$-action outside this subset.

\begin{remark}
Theorem \ref{orbit_theorem} admits the following generalization. Consider
\[
Y = \VV(u_1 \dots, u_m v_1^{k_1} \dots v_l^{k_l} - f(x)),
\]
where $m, l \geq 1$ and $k_i > 1$ for $i = 1, \dots, l$. Then the complement of the open $G$-orbit $\mathcal{O}_G$ is
\[
Y_{\reg} \setminus \mathcal{O}_G = \Rf \times \Affine_u^m \times \VV(v_1 \dots v_l) \subset X \times \Affine_u^m \times \Affine_v^l.
\]
It decomposes into a disjoint union of $G$-orbits in the following way
\[
\bigsqcup_{\substack{[p] \, \in \, \Rf / \Gf{G}{0}, \\ q \, \in \, \VV(v_1 \dots v_l)}} \Gf{G}{0} \cdot p \times \Affine^m_u \times \{q\}.
\]
\end{remark}

\subsection{Corollaries and examples}

\begin{example}\label{example_div}
    Let $X$ be the flexible Danielewski surface
    \[
    X = \VV(xy - z^2 - 1) \subset \Affine^3,
    \]
    and let $Y$ be the $2$-suspension over $X$
    \[
    Y = \VV(uv^k - yz) \subset X \times \Affine^2
    \]
    with $k > 1$. Then the complement of the open $G$-orbit on $Y$
    is a divisor consisting of $G$-fixed points.
\end{example}

\begin{proof}
    Let us determine first for which LNDs polynomial $yz$ is a local slice. If $\deg_{\delta} (yz) = 1$, then
    \[
    \deg_{\delta}y + \deg_{\delta}z = 1, \;\;\;\; \deg_{\delta}x + \deg_{\delta}y = 2 \deg_{\delta}z,
    \]
    which yields the unique solution
    \[
    \deg_{\delta}x = 1, \, \deg_{\delta}y = 0, \, \deg_{\delta}z = 1.
    \]
    Every LND with such degrees is a replica of $\delta_0$, where
    \[
    \delta_0(x) = 2z, \;\; \delta_0(z) = y, \;\; \delta_0(y) = 0.
    \]
    Hence,
    \[
    \mathcal{D}_{yz} = \VV(yz, \delta_0(yz)) = \VV(yz, y^2) = \VV(y).
    \]

    Let us write down the Jacobian matrix $J$ of $\VV(yz) \subset X$:
    \[
    J = \begin{pmatrix}
        y & x & -2z\\
        0 & z & y
    \end{pmatrix}.
    \]
    If $y = 0$, then the defining equation of $X$ implies that $z \neq 0$. Therefore $\rk J = 2$ and consequently
    \[
    \mathcal{R}_{yz} = \mathcal{D}_{yz} \cap \operatorname{Reg}_X \VV(yz) = \VV(y) \cap \VV(yz) = \VV(y).
    \]

    Thus, the complement of the open $G$-orbit on the suspension is a divisor
    \[
    Y \setminus \mathcal{O}_G = \VV(y) \times \Affine \times \{0\}.
    \]
    Moreover, all $G$-orbits on this divisor are trivial. Indeed, it is easy to verify that $\mathcal{N}^{(0)}_{yz} = \{0\}$. That means $\Gf{G}{0}$ act trivially on $\mathcal{R}_{yz}$.
\end{proof}

\begin{corollary}
    Let $X$ be a flexible variety and let $f$ be a local slice of $X$. Each of the following conditions is sufficient for the flexibility of
    \[
    Y = \VV(u_1 \dots u_m v_1^{k_1} \dots v_l^{k_l} - f(x)) \subset X \times \Affine_u^m \times \Affine_v^l.
    \]
    \begin{enumerate}
        \item $\Rf = \emptyset$.
        \item $f$ is a slice of $X$.
        \item $f = h g \delta_0 g + c$ for some $c \in \KK \setminus \{0\}$ and $\delta_0 \in \LND(X)$, $h, g \in \KKX$ satisfying $\deg_{\delta_0} g = 1$ and $h \in \Ker \delta_0$. Moreover, one may take $f = h g p + c$, where $p$ satisfies $\exists \, n : p^n = \delta_0 g$.
    \end{enumerate}
\end{corollary}

\begin{proof}
    \begin{enumerate}
        \item This is an immediate consequence of the theorem.
        \item If $f$ is a slice of $X$, then $\Df = \Rf = \emptyset$.
        \item Let $f = h g p + c$, where $h, g, p,$ and $c$ are as in the statement. Since $p^n = \delta_0 g \in \Ker \, \delta_0$, factorial closedness of the kernel of an LND implies that $p \in \Ker \, \delta_0$. It is straightforward to verify that under this assumption $f$ is a local slice of $X$.
        
        Let us compute the radical of the ideal $\If$:
        \[
        \sqrt{\If} \supseteq \sqrt{(f, \delta_0 f)} = \sqrt{(h g p + c, h p \delta_0 g)} \ni hp.
        \]
        Since the radical contains $hp$, it also contains the constant
        \[
        (hgp + c) - g(hp) = c.
        \]
        Therefore, $\Df = \emptyset$.
    \end{enumerate}
\end{proof}

\begin{remark}
    When $f$ is a slice, the flexibility of $Y$ can also be proved by an alternative argument. By the Slice Theorem,
    \[
    X \cong Z \times \Affine_f,
    \]
    where $f$ is the coordinate on the affine line $\Affine_f$. Hence,
    \[
    \Susp(Z \times \Affine_f, f, 1, k) \cong Z \times \Affine^2 \cong X \times \Affine.
    \]
    If $X$ is flexible, then so is $X \times \Affine$. \emph{Thus, the theorem on $G$-orbits yields a genuinely new result only when $f$ is a local slice but not a slice.}
\end{remark}

The last part of the corollary is useful because it provides a simple and constructive way to produce examples of flexible suspensions. This is well illustrated by the following example.

\begin{example}
    Let $X$ be defined by
    \[
    X = \VV(xy - z^n - f_{n - 1}(y) z^{n - 1} - \dots - f_1(y) z - f_0(y)) \subset \Affine^3
    \]
    for arbitrary polynomials $f_0(y), \dots, f_{n - 1}(y) \in \KK[y]$ and $n \ge 1$. It is well known that such $X$ is a flexible Danielewski surface. Consider the $(m + l)$-suspension
    \[
    Y = \VV(X, u_1 \dots u_m v_1^{k_1} \dots v_l^{k_l} - zyf(y) - c).
    \]
    Then it is flexible for any $f(y) \in \KK[y]$, $c \in \KK \setminus \{0\}$, $m, l \ge 1$, and $k_j > 1$.
\end{example}

\begin{proof}
    This follows from the last part of the corollary by considering the derivation
    \[
    \delta_0(x) = \frac{\partial}{\partial z} \left(z^n + \sum_{i = 0}^{n - 1} f_i(y) z^i\right), \;\; \delta_0(z) = y, \;\; \delta_0(y) = 0
    \]
    with $h = f(y)$ and $g = z$.
\end{proof}

Since the construction of $f = h g \delta_0 g + c$ only requires choosing a nontrivial $\delta_0 \in \LND(X)$ together with an arbitrary local slice $g$ of $\delta_0$ and an arbitrary element $h \in \Ker \, \delta_0$, we obtain the following proposition.

\begin{proposition}
    Let $X$ be a flexible affine algebraic variety. Then there exists an element $f \in \KKX \setminus \KK$ such that $\Susp(X, f, 1, k_2, \dots, k_m)$ is flexible for any $k_i > 1$, $i = 2, \dots, m$.
\end{proposition}

\subsection{Auxiliary lemmas}

Henceforth, let $X$ be a flexible affine algebraic variety, $f$ an arbitrary local slice on $X$, and $Y = \Susp(X, f, 1, k)$ an irreducible $2$-suspension with the weight $k > 1$.

We formulate several lemmas based on the lemmas from \cite{Arzhantsev-suspensions}, which were used to prove the flexibility of $\Susp(X, f, 1, 1)$. It should be noted that the statements and proofs of these lemmas undergo substantial simplification. This is due to the transition from proving the infinite transitivity of the $\SAut$-action on $\Susp(X, f, 1, 1)$ to our setting, where our only goal is to determine the $G$-orbits. We also note that at the time \cite{Arzhantsev-suspensions} was written, the equivalence between transitivity and infinite transitivity of the $\SAut$-action on varieties had not yet been proved (although it was conjectured to hold). This equivalence was established later in \cite{Arzhantsev-flexible}.

\begin{lemma}\label{lemma_reg}
Let $\pi: Y \rightarrow X$ denote the restriction of the projection $X \times \Affine^2 \rightarrow X$ to $Y$. Then $\pi(Y_{\reg}) = X_{\reg}$.
\end{lemma}

\begin{proof}
Let $f_1, \dots, f_r \in \KK[x_1, \dots, x_s]$ be functions generating the ideal of $X$. A point $P \in X$ is regular only if the rank of the Jacobian matrix
\[
J_X = \begin{pmatrix}
\frac{\partial f_1}{\partial x_1} & \frac{\partial f_1}{\partial x_2} & \dots & \frac{\partial f_1}{\partial x_s}\\
\frac{\partial f_2}{\partial x_1} & \frac{\partial f_2}{\partial x_2} & \dots & \frac{\partial f_2}{\partial x_s}\\
\dots & \dots & \dots & \dots\\
\frac{\partial f_r}{\partial x_1} & \frac{\partial f_r}{\partial x_2} & \dots & \frac{\partial f_r}{\partial x_s}
\end{pmatrix}
\]
attains its maximal value $s - \dim \, X$ at this point.

The Jacobian matrix of $Y$ is
\[
J_Y = \begin{pmatrix}
\frac{\partial f_1}{\partial x_1} & \frac{\partial f_1}{\partial x_2} & \dots & \frac{\partial f_1}{\partial x_s} & 0 & 0\\
\frac{\partial f_2}{\partial x_1} & \frac{\partial f_2}{\partial x_2} & \dots & \frac{\partial f_2}{\partial x_s} & 0 & 0\\
\dots & \dots & \dots & \dots & 0 & 0\\
\frac{\partial f_r}{\partial x_1} & \frac{\partial f_r}{\partial x_2} & \dots & \frac{\partial f_r}{\partial x_s} & 0 & 0\\
\frac{\partial f}{\partial x_1} & \frac{\partial f}{\partial x_2} & \dots & \frac{\partial f}{\partial x_s} & -v^k & -kuv^{k - 1}
\end{pmatrix}.
\]

Since $\rk \, J_Y \leq \rk \, J_X + 1$ and $\dim Y = \dim X + 1$, we have $\pi(Y_{\reg}) \subseteq X_{\reg}$. On the other hand, the reverse inclusion also holds. For an arbitrary point $P \in X_{\reg}$, we have $\rk \, J_X(P) = s - \dim \, X$. There exists a square submatrix $M(P)$ of $J_X(P)$ whose rank and order are both equal to $s - \dim \, X$. Extend this matrix by adding the last row and the penultimate column of $J_Y$ for some $v \neq 0$. The resulting matrix $M^{'}(P)$ has rank greater by one at the point $P^{'} = (P, u, v)$, where $u = f(P)/v^k$. Hence $P^{'} \in Y_{\reg}$ and $\pi(P^{'}) = P$.
\end{proof}

\begin{corollary}\label{f_reg}
\begin{enumerate}
\item $(P, u, 0) \in Y_{\reg} \iff P \in \RegVf$.
\item Let $v \neq 0$. Then $(P, 0, v) \in Y_{\reg} \iff P \in \VV(f) \cap X_{\reg}$.
\end{enumerate}
\end{corollary}

\begin{proof}
In both statements, the inclusion $\pi(P^{'}) = P \in \VV(f)$ follows immediately from the suspension equation: $(uv^k - f)(P^{'}) = -f(P) = 0$. The proof of the second statement is analogous to the proof of the lemma, so we focus only on the first one.

At the point $P^{'} = (P, u, 0) \in Y_{\reg}$ the rank $\rk \, J_Y(P^{'}) = s + 1 - \dim X = s - \dim \VV(f)$. Since $J_Y(P^{'})$ does not depend on $u$ in any way, we have $\rk \, J_Y(P^{'}) = \rk \, J_{\VV(f)}(P)$, and the relation for the rank implies that $P \in \RegVf$.
\end{proof}

For an arbitrary $c \in \KK$ define the hypersurfaces $U_c = \{u = c\}$ and $V_c = \{v = c\}$.

\begin{lemma}\label{Vc}
Suppose that $\SAut(X)$ acts transitively on $X_{\reg}$. Then for any $c \in \KK \setminus {0}$ the group $G_u$ acts transitively on $V_c \cap Y_{\reg}$.
\end{lemma}

\begin{proof}
Consider arbitrary points $P^{'}$ and $Q^{'}$ in $V_c \cap Y_{\reg}$. The map $\pi$ sends $V_c$ isomorphically onto $X$, and $\pi(V_c \cap Y_{\reg}) \subset X_{\reg}$. The point $P^{'}$ can be written as $(P, u, v)$, where $P = \pi(P^{'}), u = u(P^{'}) = f(P) / c^k$. Accordingly, every regular point of $X$ has a preimage in $Y_{\reg}$. Thus, $\pi$ induces an isomorphism $X \cong V_c$, sending the point $P$ to $P^{'}$.

By the transitivity of the $\SAut(X)$-action, there exists an automorphism $\psi \in \SAut(X)$ sending $P$ to $Q$. It can be written as
\[
\psi = \prod_{i = 1}^r \exp(\delta^{(i)}),
\]
for some LNDs $\delta^{(i)}$. Let $\hat{\delta}^{(i)}_u = \hat{\delta}_u(\delta^{(i)}, \alpha_i)$ be their extensions, with suitable constants $\alpha_i$, so that the restrictions of these LNDs to $X$ coincide with $\delta^{(i)}$. Namely, if $\delta^{(i)} (f) \neq 0$, then we take $\alpha_i = c^{-k}$; otherwise, we take $\alpha_i = 1$. Then $\psi$ lifts to the automorphism
\[
\hat{\psi} = \prod_{i = 1}^r \exp(\hat{\delta}^{(i)}_u) \in G_u \subset \SAut(Y),
\]
which sends $P^{'}$ to $Q^{'}$.
\end{proof}

\begin{corollary}\label{out_U0}
Suppose that $\SAut(X)$ acts transitively on $X_{\reg}$. Fix $c \in \KK \setminus {0}$. Then for every point $P^{'} \in U_0 \cap V_c \cap Y_{\reg}$ there exists an automorphism $\psi \in \SAut(Y)$ such that $\psi(P^{'}) \in (V_c \cap Y_{\reg}) \setminus U_0$.
\end{corollary}

\begin{proof}
The group $G_u$ acts transitively on $V_c \cap Y_{\reg}$, so there exists an automorphism $\psi \in G_u$ sending $P^{'}$ to a point $\psi(P^{'}) \in (V_c \cap Y_{\reg}) \setminus U_0$. The latter set is nonempty. Indeed, to an arbitrary regular point $P \in X_{\reg} \setminus \VV(f)$ we associate the point $P^{'} = (P, c, u) \in Y$ with coordinate $u = \frac{f(P)}{c^k} \neq 0$. For this point $P^{'}$, the rank of the Jacobian matrix $J_Y(P^{'})$ is maximal and equal to $s + 1 - \dim X$. Hence $P^{'} \in (V_c \cap Y_{\reg}) \setminus U_0$.
\end{proof}

The proof of Lemma \ref{Vc} yields the following useful observation about $\RegVf$.

\begin{corollary}
$\RegVf$ is invariant under the action of $\Gf{G}{0}$ on $\VV(f)$.
\end{corollary}

\begin{proof}
Choose a point $P \in \RegVf$ and show that $\psi(P) \in \RegVf$ for any $\psi \in \Gf{G}{0}$ acting on the divisor $\VV(f)$. To this end, write $\psi$ as the product
\[
\psi = \prod_{i = 1}^r \exp(\delta^{(i)}), \;\; \delta^{(i)} \in \Gf{N}{0}.
\]
By Corollary \ref{f_reg}, we have $\RegVf = \pi(U_d \cap V_0 \cap Y_{\reg})$ for any fixed $d \in \KK$. Lift the automorphism to $U_d \cap V_0 \cap Y_{\reg}$ by extending each $\delta^{(i)}$ to $\delta^{'(i)}$:
\[
\hat{\psi} = \prod_{i = 1}^r \exp(\delta^{'(i)}).
\]
The automorphism $\hat{\psi}$ sends the regular point $(P, d, 0) \in Y_{\reg}$ to the regular point $(\psi(P), d, 0) \in Y_{\reg}$. Projecting back to $X$, we obtain $\psi(P) \in \RegVf$.
\end{proof}

\begin{corollary}\label{rf_inv}
$\Rf$ is invariant under the action of $\Gf{G}{0}$ on $\VV(f)$.
\end{corollary}

\begin{lemma} \label{along_U}
Suppose that $\SAut(X)$ acts transitively on $X_{\reg}$. Then for every $u_0 \in \KK$ there exists an automorphism $\psi \in G_u$ sending the point $P^{'} = (P, 0, 0) \in Y_{\reg}$ to $\psi(P^{'}) = (P, u_0, 0) \in Y_{\reg}$, provided that both points exist.
\end{lemma}

\begin{proof}
For the point $P^{'} = (P, 0, 0) \in Y_{\reg}$, its projection $P$ belongs to $\RegVf$ (see Corollary \ref{f_reg}), and $df(P) \neq 0$ in the cotangent space $T_P^* X$. Since $X$ is flexible, there exists $\partial \in \LND(X)$ such that $\partial(f)(P) \neq 0$. Extending it, we obtain the derivation
\[
\hat{\partial} = v^k \partial + \partial(f) \partial_u, \;\;\;\; \hat{\partial}(u)(P^{'}) = \partial(f) \neq 0.
\]
The orbit of the corresponding $\Ga$-subgroup is $(P, \KK, 0)$.
\end{proof}

\subsection{Proof of $G$-orbit decomposition theorem}\label{proof_theorem_orbits}

Fix a point $Q^{'} \in V_c \cap Y_{\reg}$, where $c \in \KK \setminus {0}$, and choose a point $P^{'} \in Y_{\reg}$ such that the conditions $v(P^{'}) = 0$ and $\pi(P^{'}) \in \Df$ do not hold simultaneously. We show that there exists an automorphism in $G$ sending $P^{'}$ to $Q^{'}$. As usual, let $P$ denote the projection $\pi(P^{'})$.

1. If $P \notin \Df$, then there exists $\delta \in \Gf{N}{1}$ such that $\delta(f)(P) \neq 0$. Extend $\delta$ as
\[
\hat{\delta} = k u v^{k - 1} \delta + \delta (f) \, \partial_v, \;\;\;\; \hat{\delta} (v) (P^{'}) = \delta (f)(P) \neq 0.
\]
The orbit of the point $P^{'}$ under the $\Ga$-subgroup $\exp(\KK \hat{\delta})$ intersects $V_c$. By Lemma \ref{Vc}, the action of $G$ on $V_c \cap Y_{\reg}$ is transitive, so there exists an automorphism sending the intersection point to $Q^{'}$.

2. Suppose that $P \in \Df$. If $v(P^{'}) = d \neq 0$, then by Corollary \ref{out_U0} there exists an automorphism $\psi$ such that $\psi(P^{'}) \in (V_d \cap Y_{\reg}) \setminus U_0$. For such a point, $\pi(\psi(P^{'})) \notin \VV(f)$, and therefore $\pi(\psi(P^{'})) \notin \Df$. Thus, we are reduced to the first case.

It remains to consider the final case, which corresponds to the complement of the open orbit.

3. Suppose that $P \in \Df$ and $v(P^{'}) = 0$. By Corollary \ref{f_reg}, we have $P \in \Rf$. From now on, assume that $\Rf$ is nonempty; otherwise, the proof is complete. Restrict the extensions of all $\delta \in \LND(X)$ to $Z = \Rf \times \Affine^1_u \times \{0\}$:
\[
\forall \, \delta \notin \Gf{N}{0}: \hat{\delta}_u(\delta) |_Z = \delta (f) \, \partial_u,
\]
\[
\forall \, \delta \in \Gf{N}{1}: \hat{\delta}_v(\delta) |_Z \equiv 0,
\]
\[
\forall \, \delta \in \Gf{N}{0}: \hat{\delta}(\delta) |_Z \equiv \delta^{'}.
\]
Since $\Rf$ is invariant under the action of $\Gf{G}{0}$, the subset $Z$ is invariant under the action of $\Ga$-subgroups corresponding to the extended LNDs. Hence $Z$ is $G$-invariant. By Lemma \ref{along_U}, all points of the form $(P, \KK, 0)$ belong to the same $G$-orbit. It is now easy to see that the $G$-orbits on $Z$ are precisely $\Gf{G}{0}p \times \Affine^1_u \times \{0\}$, where $p \in \Rf$.

\end{document}